\documentclass[11pt, a4paper]{amsart}

\usepackage{geometry} 
\usepackage{fancyhdr}
\usepackage{url}
\usepackage{amssymb,amsthm}
\usepackage{amsmath}
\usepackage{accents}
\usepackage{eucal}
\usepackage{amscd}
\usepackage[shortalphabetic]{amsrefs}
\usepackage{mathptmx}
\usepackage{times}
\usepackage{graphics}
\usepackage{units}

\allowdisplaybreaks[1]

\newtheorem{theorem}{Theorem}[section]
\newtheorem{lemma}[theorem]{Lemma}
\newtheorem{corollary}[theorem]{Corollary}

\newtheorem{proposition}[theorem]{Proposition}

\newtheorem{thm}{Theorem}[section]

\newtheorem{lem}[thm]{Lemma}
\newtheorem{prop}[thm]{Proposition}

\newcommand{\circo}{\accentset{\circ}}

\providecommand{\abs}[1]{\lvert#1\rvert}

\DeclareMathOperator{\tr}{tr}
\DeclareMathOperator{\diam}{diam}
\newcommand{\ho}{\accentset{\circ}{A}}

\newcommand{\la}{\langle}
\newcommand{\ra}{\rangle}
\newcommand{\s}{\gamma}

\newcommand{\p}{\partial}

\newcommand{\3}{1}
\newcommand{\4}{2}

\newcommand{\e}{\epsilon}
\newcommand{\al}{\alpha}

\newcommand{\Rp}{Rm^{\perp}}

\newcommand{\Acirc}{\accentset{\circ}{A}}
\newcommand{\mc}{\mathcal}
\DeclareMathOperator{\Div}{div}
\newcommand{\R}{\mathbb{R}}
\newcommand{\mbb}{\mathbb}

\newcommand{\Kp}{K^{\perp}}

\begin{document}

\title{Evolving Pinched Submanifolds of the Sphere by Mean Curvature Flow}
\author{Charles Baker}
\address{Australian eHealth Research Centre\\
Level 5 - UQ Health Sciences Building 901/16, Royal Brisbane and Women's Hospital\\
Herston  Qld  4029\\
Australia}
\email{Charles.Baker@csiro.au}

\author{Huy The Nguyen}
\address{School of Mathematical Sciences\\ 
Queen Mary University of London\\ 
Mile End Road\\ 
London E1 4NS, UK}
\email{h.nguyen@qmul.ac.uk}

\maketitle

\begin{abstract} 
In this paper, we prove convergence of the high codimension mean curvature flow in the sphere to either a round point or a totally geodesic sphere assuming a pinching condition between the norm squared of the second fundamental form and the norm squared of the mean curvature and the background curvature of the sphere. We show that this pinching is sharp for dimension $n\geq 4$ but is not sharp for dimension $n=2,3$. For dimension $n=2$ and codimension $2$, we consider an alternative pinching condition which includes the normal curvature of the normal bundle. Finally, we sharpen the Chern-do Carmo-Kobayashi curvature condition for surfaces in the four sphere - this curvature condition is sharp for minimal surfaces and we conjecture it to be sharp for curvature flows in the sphere.   
\end{abstract}
\section{Introduction}
We concern ourselves with submanifolds of the sphere evolving with velocity equal to mean curvature and seek to understand the widest class of initial submanifolds that converge to totally geodesic submanifolds or shrink to points in finite time. The equivalent problem for hypersurfaces of the sphere has been investigated by Huisken \cite{Huisken1987}. Mean curvature flow of hypersurfaces has been intensively studied since Huisken's seminal result on the flow of convex surfaces, but in contrast, progress in high codimension has been slow, dogged by the presence of normal curvature. Recently, the authors made a breakthrough for mean curvature flow of two surfaces of codimension two in a Euclidean background \cite{Baker2017}, and showed that by explicitly including normal curvature in the pinching cone, the equivalent hypersurface result could almost be obtained. The presence of normal curvature in high codimension creates additional unfavourable reaction terms driving singularity formation. As our recent paper demonstrates, at least in codimension two, the additional unfavourable reaction terms can potentially be controlled by careful incorporation of normal curvature into the initial pinching cone. The main results presented in this paper are enabled by an improved understanding of how to control normal curvature along the mean curvature flow. Herein, we present three new results, the first two concerning mean curvature flow of submanifolds of the sphere subject to different initial pinching conditions (the second involving normal curvature), and the third a classification of submanifolds of the sphere with pointwise pinched intrinsic and normal curvatures:
\begin{thm}\label{thm:MainThm1}
Let $\Sigma_0^n = F_0(\Sigma^n)$ be a closed submanifold smoothly immersed in $\mathbb{S}^{n+k}$.  If $\Sigma_0$ satisfies
\begin{equation*}
	\begin{cases}
		\abs{A}^2 \leq \frac{4}{3n}\abs{H}^2 + \frac{n}{2} \bar K, \quad n = 2,3 \\
		\abs{A}^2 \leq \frac{1}{n-1}\abs{H}^2 + 2\bar K, \quad n \geq 4,
	\end{cases}
\end{equation*}
then either
\begin{enumerate}
	\item MCF has a unique, smooth solution on a finite, maximal time interval $0 \leq t < T < \infty$ and the submanifold $\Sigma_t$ contracts to a point as $t \rightarrow T$; or
	\item MCF has a unique, smooth solution for all time $0 \leq t < \infty$ and the submanifold $\Sigma_t$ converges to a totally geodesic submanifold $\Sigma_{\infty}$.
\end{enumerate}
\end{thm}

\begin{thm}\label{thm:MainThm2}
Suppose $\Sigma_0=F_0(\Sigma^2)$ is a closed surface smoothly immersed in $\mathbb{S}^{4}$.  If $\Sigma_0$ satisfies $\abs{A}^2 +2\gamma|K^{\bot}| \leq k\abs{H}^2+4( k - \frac 1 2)\bar{K}$, where $ \gamma = 1 - \nicefrac43 k$ and $k \leq \nicefrac{29}{40}$,
then the mean curvature flow of $\Sigma_0$ has a unique smooth solution $\Sigma_t$ on a maximal time interval $t \in [0,T)$. If $T$ is finite then there exists a sequence of rescaled mean curvature flows $F_j : \Sigma^2 \times I_j \rightarrow \mathbb{R}^{4}$ containing a subsequence of mean curvature flows (also indexed by $j$) that converges to a limit mean curvature flow $F_{\infty} : \Sigma^2_{\infty} \times (-\infty, 0] \rightarrow \mathbb{R}^{4}$ on compact sets of $\mathbb{R}^{4} \times \mathbb{R}$ as $j \rightarrow \infty$. Moreover, the limit mean curvature flow is a shrinking sphere. If $T=\infty$ then the flow converges to a totally geodesic sphere. 
\end{thm}

\begin{thm}\label{thm:MainThm3}
Suppose a two surface $\Sigma^2$ minimally immersed in $\mbb{S}^4$ satisfies $|K^\perp| \leq 2 |K|$. Then either 
\begin{enumerate}
	\item $|A| ^ 2 \equiv 0 $ and the surface is a geodesic sphere; or
	\item $|A| ^ 2 \not \equiv 0 $, in which case either
	\begin{enumerate}
		\item $|K^{\perp}| = 0$ and the surface is the Clifford torus, or
		\item $K^{\perp} \neq 0$ and it is the Veronese surface.
	\end{enumerate}
\end{enumerate}
\end{thm}

The first theorem appeared in the first author's doctoral thesis, which has not yet been published in peer-reviewed form, with the less optimal constant $2 (n-1)/3$ preceding the background curvature. The proof presented here differs the first author's doctoral thesis by improving upon the pinching constant and by replacing the complicated Stampaccia iteration by a more elegant blow-up argument using a new pointwise gradient estimate; more on this is said below shortly. We briefly mention that the improvement of the constant from $2(n-1)/3 \bar K$ to $ n/2 \bar K$ can also be achieved by making use of the discovery made by \cite{Huisken2009} that the nonlinearity in the Simons identity need only be positive to the highest order of mean curvature in order for the Stampacchia iteration argument to work.

The main theorem of \cite{Andrews2010} is optimal for submanifolds of dimension four and greater (independent of the codimension), where the tori $\mathcal{S}^{n-1}(\epsilon) \times \mathcal{S}(1) \subset \mathbb{S}^{n} \times \mathbb{S}^{2}$ are obstructions to improving the pinching constant beyond $1/(n-1)$. The theorem is suboptimal in dimensions two and three, with pinching constant $k=4/(3n)$, because of unfavourable reaction terms. In a recent breakthrough, for codimension two surfaces in a Euclidean background, the authors were able to improve this constant from $4/(3n)$ to $29/40$ by including the normal curvature in the pinching cone. Theorem \ref{thm:MainThm2} extends the result of \cite{Baker2017} to submanifolds of a spherical background. With the inclusion of normal curvature, the new pinching condition turns out to be optimal for the reaction terms, but the gradient terms still obstruct the attainment of optimal pinching, analogous to the flow of two dimensional hypersurfaces in a spherical background \cite{Huisken1987}. In the hypersurface case, Huisken achieves a constant of $3/4$ whereas for codimension two surfaces we achieve $3/4 - 1/40$. In both cases the constants are determined the gradient terms; in the codimension two case, the term $1/40$ appears in order to accommodate the gradient of normal curvature. We do not know whether the pinching constant $29/40$ can be extended to $3/4$. We conjectured in \cite{Baker2017} that (for two surfaces of codimension two in a Euclidean background) the true obstruction to the theorem is the Clifford torus immersed in $\mathbb{R}^4$, corresponding to the pinching cone $|A|^2 < |H|^2$. In the present case, we conjecture that the true obstruction to Theorem \ref{thm:MainThm2} is the Veronese surface, a minimal submanifold of the four sphere.

The third result we present, Theorem \ref{thm:MainThm3}, is a new classification of minimal submanifolds of the sphere, made possible by our exact computation of the Simons identity nonlinearity. The Simons identity plays a key role in a series of classification results initiated in a famous paper by Chern, do Carmo and Kobayashi \cite{ChernCarmoKobayashi1970}, later extended to encompass submanifolds of the sphere with parallel mean curvature by Santos \cite{Santos1994}. With our refined understanding of the Simons identity nonlinearity we are able to provide a new classification result depending not on the length of the second fundamental form, but rather on a pointwise pinching of the intrinsic and normal curvatures. Combined with a careful analysis of the curvature terms, the proof is achieved by an application of the strong maximum principle, mimicking modern proofs of Simons' famous result \cite{Simons1968}.



The outline of this paper is as follows. The broad arc of the proof is the same as \cite{Huisken1984}, however the proof develops by a more efficient series of estimates, completely avoiding the Poincar\'e-type inequality constructed by painful estimation of the Simons identity nonlinearity and the Stampacchia iteration. After proving that curvature pinching is preserved along the flow in section \ref{s: Preservation of curvature pinching}, in place of the Stampacchia iteration, in section \ref{s:gradientEstimate} we prove a pointwise gradient estimate and blow-up argument to characterise the shape of the evolving submanifold at a finite time singularity. The case of infinite lifespan is also treated, giving rise to the second case in theorem \ref{thm:MainThm1}. We then move on to consider the case of evolving two surfaces in the four-sphere, and extend the normal curvature-pinched cone introduced in \cite{Baker2017} to a spherical background and prove its preservation along the flow in section \ref{s:twoSurfaces}. The remainder of the proof is similar to that presented in section \ref{s:gradientEstimate} (or by the Stampacchia iteration argument in \cite{Baker2017}) with straightforward adjustments and is not duplicated. In the final section, we present our refinement of the famous theorem of Chern, do Carmo and Kobayashi. The exact estimation of the Simons identity nonlinearity was used (in a Euclidean background) to derive the Poincar\'{e}-type inequality used in the Stampacchia iteration argument in \cite{Baker2017}.

\textbf{Acknowledgements.}  
The second author would like to acknowledge the support of the EPSRC through the grant EP/S012907/1.

\section{The evolution equations in a sphere}\label{s: The evolution equations in a sphere}
The geometric evolution equations for high codimension mean curvature flow in an arbitrary Riemannian background were derived in \cite{Andrews2010}. We recall the evolution of the second fundamental form is given by
\begin{align*}
\nabla_{\partial_t}h_{ij} &= \Delta h_{ij}+h_{ij}\cdot h_{pq}h_{pq}+h_{iq}\cdot h_{qp}h_{pj}
+h_{jq}\cdot h_{qp}h_{pi}-2h_{ip}\cdot h_{jq} h_{pq}\notag\\
&\quad +2\bar R_{ipjq}h_{pq}-\bar R_{kjkp}h_{pi}-\bar R_{kikp}h_{pj}+h_{ij\alpha}\bar R_{k\alpha k\beta}\nu_\beta\notag\\
&\quad -2h_{jp\alpha}\bar R_{ip\alpha\beta}\nu_\beta-2h_{ip\alpha}\bar R_{jp\alpha\beta}\nu_\beta
+\bar\nabla_k\bar R_{kij\beta}\nu_\beta-\bar\nabla_i\bar R_{jkk\beta}\nu_\beta.\label{eq:reactdiffuse}
\end{align*}
We make use of shorthand notation for the reaction terms
\begin{align*}
R _1 = \sum_{\alpha,\beta} \bigg( \sum _{i ,j} h _{ij \alpha} h _{ij \beta} \bigg) ^ 2 + | Rm ^ \perp | ^ 2, \quad R_2 = \sum_{i,j} \bigg( \sum _{\alpha} H _ \alpha h_{ij \alpha}\bigg)^ 2,
\end{align*}
where $ R_{ij\alpha \beta}^{\perp} = h_{ip\alpha}h_{jp\beta} - h_{jp\alpha}h_{ip\beta}$.

The evolution equations can be simplified in background spaces of constant curvature such as the sphere. Suppose $\partial_a$, $1 \leq a \leq {n+k}$ is an orthogonal local frame for background sphere of constant curvature $\bar K$. In such a frame the curvature tensor of the sphere is
\begin{equation}\label{e: curv tensor of sphere}
	\bar R_{abcd} = \bar{K}( \delta_{ac}\delta_{bd} - \delta_{ad}\delta_{bc} ).
\end{equation}
Using this as well as the fact that the derivates of the constant background curvature are zero, we see the evolution for the second fundamental form is
\begin{align*}
\nabla _{\partial _ t}  h _{ij \alpha} = \triangle h _{ij\alpha} + R_{ij \alpha} + 2 \bar K H _ \alpha  g _{ij} -n \bar K h _{ij \alpha}
\end{align*}
where
\begin{align*}
 R_{ij\alpha} = h_{ij \beta}\cdot h _{pq \beta} h _{pq \alpha} + h _{iq \beta} \cdot h_{qp \beta} h _{pj \alpha}  + h _{jq \beta} \cdot h _{qp \beta} h _{pi \alpha}  - 2 h _{ip\beta} \cdot h _{jq \beta} h_{pq \alpha}.
\end{align*}
Tracing this expression with respect to $i,j$ we get 
\begin{align*}
\nabla _{\partial_t} H_ \alpha & = \triangle H _ \alpha + R _ \alpha  + n \bar K H_\alpha 
\quad \text{
where} \quad 
R_\alpha = H_\beta \cdot h_{ij \beta} h_{ij\alpha}.
\end{align*}
With a further line of computation, we see the evolution equation for $\abs{A}^2$ is given by  
\begin{equation}\label{e: evol h2 sphere}
	\frac{\partial}{\partial t}\abs{A}^2 = \Delta\abs{A}^2 - 2\abs{\nabla A}^2 + 2R_1 + 4\bar K\abs{H}^2 - 2n\bar K\abs{A}^2,
\end{equation}
and $|H|^2$ by
\begin{equation}\label{e: evol H2 sphere 2}
	\frac{\p}{\p t} \abs{H}^2 = \Delta\abs{H}^2 - 2\abs{\nabla H}^2 + 2 R_2 + 2n\bar{K}\abs{H}^2.
\end{equation}
The evolution of normal curvature is first computed in \cite{BakerNguyen2016} where it is found to be
\begin{equation}\label{eqn:evolNormalCurv}
\begin{split}
\frac{\partial}{\partial t}   R ^{\perp} _{ij\alpha \beta} &=\Delta R ^{\perp} _{ij \alpha \beta} - 2 \sum _{p,r}\left ( \nabla_q h _{ip\alpha} \nabla _ q h _{jp\beta} - \nabla _ q h _{jp \alpha} \nabla _ q h _{i p \beta}\right ) \\
&\quad+R_{ip\alpha}h_{jp\beta} + h_{ip\alpha}R_{jp\beta} - R_{jp\alpha}h_{ip\beta} - h_{jp\alpha}R_{ip\beta}- n \bar K R_{ij\alpha \beta}^\perp\\
\end{split}
\end{equation}

The contracted form of Simons' identity takes the form
\begin{equation}\label{e: contracted Simons' identity}
	\frac{1}{2} \Delta \abs{\ho}^2 = \ho_{ij} \cdot \nabla_i\nabla_j H + \abs{\nabla \ho}^2 + Z + n\bar K\abs{\ho }^2,
\end{equation}
where again
\begin{equation*}
	Z = -\sum_{\alpha, \beta} \bigg( \sum_{i,j} h_{ij\alpha} h_{ij\beta} \bigg)^2 - \abs{\Rp}^2 + \sum_{\substack{i,j,p \\ \alpha, \beta}} H_{\alpha} h_{ip\alpha} h_{ij\beta} h_{pj\beta}.
\end{equation*}
And finally, the basic gradient estimate
\begin{equation} \label{eqn: basic grad est 1 sphere} 
	\abs{\nabla A}^2 \geq \frac{3}{n+2} \abs{\nabla H}^2
\end{equation}
carries over unchanged.

\section{Preservation of curvature pinching}\label{s: Preservation of curvature pinching}
We now prove that a certain pointwise curvature pinching condition holding on the initial submanifold is preserved along the flow.

\begin{lem}
If a solution $F : \Sigma \times [0,T) \rightarrow \mathbb{S}^{n+k}$ of the mean curvature flow satisfies
\begin{equation}
	\begin{cases}
		\abs{A}^2 \leq \frac{4}{3n}\abs{H}^2 + \frac{n}{2}\bar K, \quad n = 2,3 \\
		\abs{A}^2 \leq \frac{1}{n-1}\abs{H}^2 + 2\bar K, \quad n \geq 4
	\end{cases}
\end{equation}
at $t = 0$, then this remains true as long as the solution exists.
\end{lem}
\begin{proof}
Let us consider the quadratic pinching condition $\mathcal{Q} = \abs{A}^2 - \alpha{H}^2 - \beta \bar{K}$, where $\alpha$ and $\beta$ are constants. Since we are the initial submanifold to have $H = 0$, in order to compute in a local frame for the normal bundle where $\nu_1 = H / \abs{H}$ we will consider two cases: 1) $H = 0$ and 2) $ H \neq 0$.  For the the case $H=0$, the evolution equations for $\abs{A}^2$ and $\abs{H}^2$ give us
\begin{equation}\label{e: pinch lem sphere 1}
	\frac{\p}{\p t} \mathcal{Q} = \Delta \mathcal{Q} - 2 \abs{\nabla \ho}^2 + 2 R_1 - 2n\bar{K} \abs{\ho}^2.
\end{equation}
Using the estimate of \cite{Li1992} on the normal directions of $R_1$ we have
\begin{equation*}
	R_1 = \sum_{\alpha,\beta}\bigg( \sum_{i,j} \ho_{ij\alpha}\ho_{ij\beta}\bigg)^2 + \sum_{\alpha, \beta} N(\ho_{\alpha}\ho_{\beta} - \ho_{\beta}\ho_{\alpha}) \leq \frac{3}{2} \abs{\ho}^4.
\end{equation*}
The reaction terms of \eqref{e: pinch lem sphere 1} are estimated by
\begin{equation*}
2 R_1 - 2n\bar{K} \abs{\ho}^2 \leq 3 \abs{\ho}^4 - 2n\bar{K} \abs{\ho}^2.
\end{equation*}
If $\mathcal{Q}$ is not (strictly) negative, then $\abs{\ho}^2 = \beta \bar{K}$ and
\begin{equation*}
	3 \abs{\ho}^4 - 2n\bar{K} \abs{\ho}^2 < -\beta (2n - 3\beta) \bar{K}^2
\end{equation*}
which is (strictly) negative as long as $\beta < (2/3)n$ which can not happen and the lemma follows in this case.  

Now let us consider the case $H \neq 0$.  We use the special local frames of \cite{Andrews2010}, and the evolution equation becomes
\begin{equation}\label{eqn: pinch lem sphere 2}
	\begin{split}
		\frac{\p}{\p t} \mathcal{Q} &= \Delta \mathcal{Q} - 2( \abs{\nabla A}^2 - \alpha\abs{\nabla H}^2 ) \\
			&\quad+ 2R_1 - 2\alpha R_2 - 2n \bar K\abs{\ho}^2 - 2 n ( \alpha  - 1/n ) \bar K\abs{H}^2.
	\end{split}
\end{equation}
Arguing as in Euclidean case, if $\mathcal{Q}$ does not remain (strictly) negative, we may replace $\abs{H}^2$ with $( \abs{\ho}^2 - \beta \bar{K} ) / ( \alpha - 1/n )$, and estimating as in \cite{Andrews2010},
\begin{align*}
	&2R_1 - 2\alpha R_2 - 2n \bar K\abs{\ho}^2 - 2 n ( \alpha  - 1/n ) \bar K\abs{H}^2 \\
		&\quad \leq 2\abs{\ho_1}^2 - 2( \alpha - 1/n )\abs{\ho_1}^2\abs{H}^2 + \frac{2}{n}\abs{\ho_1}^2\abs{H}^2 - \frac{2}{n}(\alpha - 1/n)\abs{H}^4 + 8\abs{\ho_1}^2\abs{\ho_-}^2 + 3\abs{\ho_-}^4 \\
		&\quad - 2n\bar K( \abs{\ho_1}^2 + \abs{\ho_-}^2 ) - 2n ( \alpha - 1/n) \bar{K} \abs{H}^2 \\
		&\quad\leq \left( 6 - \frac{2}{n(\alpha - 1/n)} \right)\abs{\ho_1}^2\abs{\ho_-}^2 + \left( 3 - \frac{2}{n(\alpha - 1/n)} \right)\abs{\ho_-}^4  \\
			&\quad + \left( 2\beta - 4n + \frac{2\beta}{n(\alpha - 1/n)} \right)\abs{\ho_1}^2 \bar K + 4 \left( \frac{\beta}{n(\alpha - 1/n)} - n \right)\abs{\ho_-}^2 \bar K \\ 
			&\quad - 2 \beta \left( \frac{\beta}{n(\alpha - 1/n)} - n \right) \bar K^2 \\
	&\quad= \left( 6 - \frac{2}{n(\alpha - 1/n)} \right)(\abs{\ho_1}^2\abs{\ho_-}^2 + \abs{\ho_-}^4 ) + \left( 2\beta - 4n + \frac{2\beta}{n(\alpha - 1/n)} \right)\abs{\ho_1}^2 \bar K \\
		&\quad- 3\abs{\ho_-}^4 + 4 \left( \frac{\beta}{n(\alpha - 1/n)} - n \right)\abs{\ho_-}^2 \bar K  - 2 \beta \left( \frac{\beta}{n(\alpha - 1/n)} - n \right) \bar K^2.
\end{align*}
In the last line we choose the coefficient of the $\abs{\ho_1}^2\abs{\ho_-}^2$ term as large as we can (that is $4/(3n)$) and we have the good term $-3\abs{\ho_-}^2$.  Since the last line above is a quadratic form, by requiring that the discriminant is negative, we have a strictly negative term.  We compute the disc1/nriminant as
\begin{equation*}
	\Delta = 8 \left( \frac{\beta}{n( \alpha - 1/n)} - n \right) \left\{2 \left( \frac{\beta}{n( \alpha - 1/n )} \right) - 3\beta \right\},
\end{equation*}
which is negative with the chosen $\alpha$ and $\beta$ in dimensions two to four.  For dimensions $n \geq 4$ the optimal value of $\alpha$ is $1/(n-1)$. Therefore with this restriction, the coefficient of $\abs{\ho_-}^4$ increases to $-2(n-4) - 3$.  The discriminant becomes
\begin{equation*}
	\Delta = 8 \left( \frac{\beta}{n( \alpha - 1/n )} - n \right) \Big\{2 \left( \frac{\beta}{n( \alpha - 1/n )} \right) - ( 2(n-4) + 3 ) \beta \Big\},
\end{equation*}
and which is strictly negative for $\beta = 2$ for all $n \geq 4$.  The most restrictive condition on the size of $\beta$ comes from the coefficient of the $\abs{\ho_1}^2 \bar K$ term, which gives the values of $\beta$ in the statement of the lemma.  With the chosen values of $\alpha$ and $\beta$ the right hand side of equation \label{eqn: pinch lem sphere 2} is strictly negative, which is contradiction, and so $\mathcal{Q}$ must stay strictly negative.
\end{proof}

\section{A gradient estimate for the second fundamental form}\label{s:gradientEstimate}
Here we establish a gradient estimate for the second fundamental form.

\begin{theorem}\label{thm_gradient}
Let $ \Sigma_t , t \in [0,T)$ be a closed $n$-dimensional quadratically bounded solution to the mean curvature flow in the round sphere of curvature $K$, $ \mathbb{S}_K^{n+k}$ with $n \geq 2$, that is 
\begin{align*}
|A|^2 - c|H|^2 - (2-\e)K <0
\end{align*}
with $ c = \frac{1}{n-1} - \eta \leq \frac{4}{3n}$.
 Then there exists a constant $ \gamma_1= \gamma_1 (n, \Sigma_0), \gamma_2 = \gamma_2 ( n , \Sigma_0)$ and $\delta_0=\delta_0(n, \Sigma_0)$ such that the flow satisfies the uniform estimate 
\begin{align*}
|\nabla A|^2 \leq (\gamma_1 |A|^4+\gamma_2)e^{\delta_0 t} \quad \text{for all $ t\in [0, T)$.}
\end{align*}

\end{theorem}

\begin{proof}
We choose here $ \kappa_n = \left( \frac{3}{n+2}-c\right)>0$. As $\frac{1}{n}\leq c\leq \frac{4}{3n}$, $n\geq 2,\kappa_n$ is strictly positive. We will consider here the evolution equation for $\frac{|\nabla A|^2}{g^2}$ where $ g = \frac{1}{n-1} |H|^2-|A|^2+ 2K $ where the initial pinching condition ensures $g$ is strictly positive. This follows since $ |A|^2-\left( \frac{1}{n-1} - \eta \right )|H|^2 +(2- \e) K  \leq  0, |H|>0$ and $\Sigma_0$ is compact, we have 
\begin{align}\label{eqn_eta}
2K + \frac{1}{n-1} |H|^2-|A|^2 >\eta | H|^2 + \e K \quad \text{or} \quad g = 2K + \frac{1}{n-1} |H|^2-|A|^2 >\eta | H|^2 + \e K>0.
\end{align}
From the evolution equations,  \eqref{eqn: pinch lem sphere 2}, we get 
\begin{align*}
\partial_t g &= \Delta g-2 \left(  \frac{1}{n-1}|\nabla H|^2-|\nabla A|^2 \right)+2  \left( c   R_2-R_1 \right)\\
 &\geq \Delta g-2  \left(  \frac{n+2}{3} \frac{1}{n-1}-1 \right) |\nabla A|^2 \\
 & \geq \Delta g+2\kappa_n  \frac{n+2}{3}| \nabla A|^2.  
\end{align*}
The evolution equation for $ |\nabla A|^2 $ is given by 
\begin{align*}
 \frac \partial{\partial t}|\nabla A|^2-\Delta |\nabla A|^2 &\leq-2 |\nabla^2 A|^2+c_n |A|^2 |\nabla A|^2 + d _n|\nabla A|^2 .
\end{align*}
Let $w,z$ satisfy the evolution equations
\begin{align*}
\frac{\partial}{\partial t}w = \Delta w+W , \quad \frac{\partial}{\partial t}z = \Delta z+Z 
\end{align*}
 then we find that 
 \begin{align*}
\partial_t \left(\frac{w}{z}\right) &= \Delta\left( \frac{w}{z}\right) +\frac{2}{z}\left \la \nabla \left( \frac{w}{z}\right) , \nabla z \right \ra+\frac{W}{z}-\frac{w}{z^2} Z\\
&= \Delta\left( \frac{w}{z}\right) +2\frac{\la \nabla w , \nabla z \ra}{z^2}-2 \frac{w|\nabla z |^2}{z^3}+\frac{W}{z}-\frac{w}{z^2} Z.
\end{align*}
Furthermore for any function $g$, we have by Kato's inequality
\begin{align*}
\la \nabla g , \nabla |\nabla A|^2 \ra &\leq 2 |\nabla g| |\nabla^2 A| |\nabla A| \leq \frac{1}{g}|\nabla g |^2 | \nabla A|^2+g |\nabla^2 A|^2.     
\end{align*}
We then get
\begin{align*}
-\frac{2}{g}| \nabla^2 A|^2+\frac{2}{g}\left \la \nabla g ,\nabla \left( \frac{|\nabla A|^2}{g}\right) \right\ra \leq-\frac{2}{g}| \nabla^2 A|^2-\frac{2}{g^3}|\nabla g|^2 |\nabla A|^2+\frac{2}{g^2}\la \nabla g ,\nabla |\nabla A|^2 \ra \leq 0 .  
\end{align*}
Then if we let $ w = |\nabla A|^2 $ and $ z = g$ with $W \leq-2 |\nabla^2 A|^2+c_n |A|^2 |\nabla A|^2 +d_n|\nabla A|^2$ and $Z\geq 2\kappa_n  \frac{n+2}{3}| \nabla A|^2 $ we get 
\begin{align*}
\frac{\partial}{\partial t}\left( \frac{|\nabla A|^2}{g}\right)-\Delta \left( \frac{|\nabla A|^2}{g}\right) &\leq \frac{2}{g}\left \la \nabla g ,\nabla \left( \frac{|\nabla A|^2}{g}\right) \right \ra+\frac{1}{g}(-2 |\nabla^2 A|^2+c_n |A|^2 |\nabla A|^2 + d_n |\nabla A|^2) \\
&-2 \kappa_n \frac{n+2}{3}\frac{|\nabla A|^4}{g^2} \\
& \leq c_n |A|^2  \frac{|\nabla A|^2}{g} + d _n \frac{|\nabla A|^2}{ g} -2 \kappa_n \frac{n+2}{3}\frac{|\nabla A|^4}{g^2}.
\end{align*}
We repeat the above computation with $w = \frac{|\nabla A|^2}{g}, z = g^{1-\sigma},$ and 
we have 
\begin{align*}
\partial_t g^{1-\sigma} = (1-\sigma)g^{-\sigma} \partial_t g&\geq (1-\sigma) g^{-\sigma} \left( \triangle g + 2 \kappa_n \frac{ n+2}{3} | \nabla A|^2 \right)
\end{align*}
Computing, note that
\begin{align*}
\Delta g^{1-\sigma} &= \Div( \nabla g^{1-\sigma} ) = \Div( (1-\sigma)g^{-\sigma} \nabla g ) = (1-\sigma) g^{-\sigma} \Delta g - \sigma(1-\sigma) g^{-1-\sigma} |\nabla  g|^2
\end{align*}
so that
\begin{align*}
(1-\sigma) g^{-\sigma} \Delta g =  \Delta g^{1-\sigma}+ \sigma(1-\sigma) g^{-1-\sigma} |\nabla  g|^2.
\end{align*}
Hence, we have 
\begin{align*}
\partial_t g^{1-\sigma} \geq \Delta g^{1-\sigma}+ \sigma(1-\sigma) g^{-1-\sigma} |\nabla  g|^2+ (1-\sigma)g^{-\sigma}\left(2 \kappa_n \frac{ n+2}{3} | \nabla A|^2 \right)
\end{align*}
So that we have for the nonlinearities, $Z \geq 0$ and  
\begin{align*}
W\leq c_n |A|^2  \frac{|\nabla A|^2}{g}-2 \kappa_n \frac{n+2}{3}\frac{|\nabla A|^4}{g^2}.
\end{align*} 

\begin{align}
\label{eqn_evol}\frac{\partial}{\partial t}\left( \frac{|\nabla A|^2}{g^{2-\sigma}}\right)-\Delta \left( \frac{|\nabla A|^2}{g^{2-\sigma}}\right) &\leq \frac{2}{g^{1-\sigma}}\left \la \nabla g^{1-\sigma} ,\nabla \left( \frac{|\nabla A|^2}{g^{2-\sigma}}\right) \right \ra  \\
\nonumber&+\frac{1}{g^{1-\sigma}}\left (   c_n |A|^2 \frac{|\nabla A|^2}{g} + d _n \frac{ |\nabla A|^2}{g}-2 \kappa_n \frac{n+2}{3}\frac{|\nabla A|^4}{g^2}\right).
\end{align}
The nonlinearity then is
\begin{align*}
\frac{|\nabla A|^2}{g^{2-\sigma}} \left( c_n|A|^2+ d_n -\frac{2 \kappa_n(n+2)}{3}\frac{|\nabla A|^2}{g}  \right).
\end{align*}
The quadratic curvature condition then bounds $ |A|^2 $ below away from zero, this can be seen from
\begin{align*}
|A|^2 &\leq \left (\frac{1}{n-1} - \eta\right) |H|^2 + (2- \e)K    \implies
\e K \leq \frac{1}{n-1}|H|^2 - |A|^2 + 2 K -\eta |H|^2 
\end{align*} 
so that $ g \geq \eta |H|^2 + \e K$, that there exists a constant $N$ so that   
\begin{align*}
Ng \geq c _n |A|^2 + d _n.
\end{align*}
Hence we have by the maximum principle, there exists a constant (with $\eta, \e $ chosen sufficiently small so that $N$ is sufficiently large that this estimate holds at the initial time) such that 
\begin{align*}
\frac{|\nabla A|^2}{g^{2-\sigma}}\leq \frac{3 N}{2 \kappa_n (n+2)}(\e K)^\sigma \leq \frac{3 N}{2 \kappa_n (n+2)}g^{\sigma}.  
\end{align*}
Therefore we see that there exists a constant $C = \frac{ 3N (\e K )^\sigma }{2 \kappa_n (n+2)}= C(n, \Sigma_0,\e,\sigma ) $ such that $\frac{|\nabla A|^2}{g^{2-\sigma}}\leq C$. 
If we let $ f = \frac{ | \nabla A| ^2 }{ g ^{2-\sigma}} $ then \eqref{eqn_evol} becomes
\begin{align*}
\partial_t f - \triangle f &\leq \frac 2 {g^{1-\sigma}} \langle \nabla g^{1-\sigma} , \nabla f \rangle  + f \left(  c_n | A|^2 + d _n - 2 \kappa_n \frac{ n+2}{3} \frac{ | \nabla A | ^2} {g }\right) 
 \end{align*}  
We then consider
\begin{align*}
\partial ( e^{ t\delta_0/2 } f ) = \frac{ \delta_0 }{2}e^{ t\delta_0/2 } f + e^{ t\delta_0/2 } \partial_t f &\leq \frac{ \delta_0}{2} e^{ t\delta_0/2 } f + e^{ t\delta_0/2 } \triangle f + \frac{2}{g} \left\langle \nabla g, \nabla (e^{ t\delta_0/2 }  f ) \right \rangle  \\
&+ e^{ t\delta_0/2 } f \left ( c_n |A|^2 + d_n - 2 \kappa_n \frac{n+2}{3} \frac{ |\nabla A| ^2}{ g} \right) 
\end{align*}
Recall that 
\begin{align*}
|A|^2 \leq \left( \frac{1}{n-1} - \eta \right)|H|^2 + (2 -\e) K 
\end{align*}
which implies 
\begin{align*}
\frac{1}{n-1} |H| ^ 2 - |A| ^2 + 2 K \geq \eta |H|^2 + \e K \leq c |A|^2 + d 
\end{align*}
for some $ c, d> 0$. This shows that there exists some $N>0$ such that
\begin{align*}
Ng \geq c_n |A|^2 + d_n + \frac{\delta_0}{2},
\end{align*}
and since $ e^{ t\delta_0/2 } \geq 1$ this implies
\begin{align*}
\frac{ | \nabla A|^2}{ g} e^{ t\delta_0/2 } \geq \frac{ |\nabla A|^2}{ g}
\end{align*} 
Hence we see that 
\begin{align*}
e^{ t\delta_0/2 } f &\left ( c_n |A|^2 + d_n + \frac{\delta_0}{2}- 2 \kappa_n \frac{n+2}{3} \frac{ |\nabla A| ^2}{ g} \right)  \leq 
e^{ t\delta_0/2 } f \left ( Ng - 2 \kappa_n \frac{n+2}{3} e^{ t\delta_0/2 } \frac{ |\nabla A| ^2}{ g} \right) \\
&\leq  e^{ t\delta_0/2 } gf \left ( N - 2 \kappa_n \frac{n+2}{3} e^{ t\delta_0/2 } \frac{ |\nabla A| ^2}{ g^2} \right).
\end{align*}
Hence we can choose $\e, \eta$ sufficiently small and applying the maximum principle we get 
\begin{align*}
e^{ t\delta_0/2 } \frac{ |\nabla A| ^2}{ g^{2-\sigma}} \leq \frac{ 3 N(\e K )^{\sigma}}{ 2 \kappa_n (n+2) }. 
\end{align*}

\end{proof}

We will also want to control the time derivative of curvature with constants with explicit dependence. In order to do so, we now derive quantitative estimates for the second derivative of curvature.   
\begin{theorem}
Let $ \mc M$ be a solution of the mean curvature flow then there exists constants $\gamma_3, \gamma_4, \delta_0$  depending only on the dimension and pinching constant so that 
\begin{align*}
|\nabla^2 A|^2 \leq (\gamma_3 |A|^6+\gamma_4)e^{-\delta_0/2 t}. 
\end{align*} 
\end{theorem}

As a special case of our estimates we get the following statement. 
\begin{corollary}\label{cor_gradient}
Let $ \mc M_t$ be a mean curvature flow. Then there exists $c^\#, H^\#,\delta_0>0$ such that for all $p \in \mc M $ and $t>0$ which satisfy
\begin{align*}
|H(p,t)| \geq H^\# \implies |\nabla H(p,t)| \leq c^\#e^{-\delta_0/2t} |H(p,t)|^2, \quad |\partial_t H(p,t)|\leq c^\#e^{-\delta_0/2} |H(p,t)|^3.
\end{align*}
\end{corollary}

Note that the following Lemma is purely a statement concerning submanifolds subject to gradient estimate for the mean curvature and is not concerned with mean curvature flow.
\begin{lemma}\label{lem_localHarnack}
Let $ F: \mc M^n \rightarrow \mathbb{S}^{n+k}$ be an immersed submanifold. Suppose there exists $ c^\#, H^\#, \delta_0>0 $ such that
\begin{align*}
|\nabla H(p)|^2 \leq c^{\#}e^{-\delta_0/2t}|H(p)|^2
\end{align*}
for any $ p\in \mc M $ such that $ |H|(p) \geq H^\#$. Let $ p_0\in \mc M$ satisfy $ |H(p_0)| \geq \gamma H^{\#}$ for some $ \gamma >1 $. Then
\begin{align*}
|H(q)| \geq \frac{|H(p_0)|}{1+c^{\#}e^{-\delta_0/2t}d ( p_0, q)|H(p_0)|}\geq \frac{|H(p_0)|}{\gamma},\quad \forall q \mid d(p_0,q) \leq \frac{\gamma-1}{c^{\#}e^{-\delta_0/2t}}\frac{1}{|H|(p_0) }.
\end{align*}

\end{lemma}
\begin{proof}
The proof is essentially that of \cite{Huisken2009} or \cite{Nguyen2018a}. 
\end{proof}

Finally we have the following
\begin{theorem}\label{thm_spherical}
Let $ F:[0,T)\times \mc M^n \rightarrow \mathbb{S}^{n+k}\subset \R^{n+k+1}$ be a smooth solution to the mean curvature flow such that $ F_0(p) = F(0,p)$ is compact and quadratically bounded. Then for all $ \e >0$ there is a $H_0>0$ such that if $ |H(p,t)| \geq H_0$ then 
\begin{align*}
\frac{|A|^{2}}{|H|^2}\leq \left (\frac{1}{n}+\e  \right).
\end{align*} 
and a $T_0>0$ such that if $ t > T_0$ 
\begin{align*}
|A|^2 \leq \e. 
\end{align*}
\end{theorem}
\begin{proof} 
The proof essentially follows that of \cite{Nguyen2018a}. Here we sketch the argument and point out the differences in the case of the sphere. Let us first consider the following case 
\begin{align*}
\lim_{t\rightarrow T_{\max}}\sup_{\mc M_t}|A(p,t) |^2 =+\infty.
\end{align*}
Furthermore, since 
$\frac{1}{n}|H|^2< |A|^2 < \frac{1}{n-1}| H|^2 + 2 K$, the second fundamental form $A$ and the mean curvature $H$ have the same blow up rate, so we must have 
\begin{align*}
\lim_{t\rightarrow T_{\max}}\sup_{\mc M_t}|H(p,t) |^2 =+\infty.
\end{align*}
Suppose the estimate is not true. Then there exists an $\e>0$ where we have 
\begin{align*}
\limsup_{t\rightarrow T_{\max}}\sup_{p \in \mc M_t}\frac{|A(p,t) |^2}{|H(p,t) |^2}= \frac{1}{n}+ \e > \frac{1}{n}.
\end{align*}
Furthermore there exists a sequence of points $ p_k $ and times $ t_k $ such that as $ k \rightarrow \infty $, $ t_k \rightarrow T_{\max}$  
and 
\begin{align*}
\lim_{k\rightarrow \infty}\frac{|A(p_k, t_k ) |^2}{|H(p_k, t_k ) |^2}= \frac{1}{n}+ \e .
\end{align*}
We perform a parabolic rescaling of $ \bar M_t^k $ in such a way that the norm of the mean curvature at $(p_k,t_k)$ becomes $n-1$. That is, if $F_k$ is the parameterisation of the original flow $ \mc M_t^k $, we let $ \hat r_k = \frac{n-1}{|H(p_k,t_k)|}$, and we denote the rescaled flow by $ \mc M_t^k $ and we define it as 
\begin{align*}
\bar F_k (p,\tau) = \frac{1}{\hat r_k}( F_k ( p , \hat r_k^2 \tau+t_k)-F_k (p_k , t_k) )
\end{align*}  

For simplicity, we choose for every flow a local co-ordinate system centred at $ p_k$. In these co-ordinates we can write $0$ instead of $ p_k$. The parabolic neighbourhoods $\mc P^k ( p_k, t_k, \hat r_k L, \hat r_k^2 \theta)$ in the original flow becomes $ \bar{\mc P}(0,0,\theta, L)$. By construction, each rescaled flow satisfies 
\begin{align*}
\bar F_k (0,0) = 0, \quad |\bar H_k (0,0) | = n-1. 
\end{align*}
The gradient estimates give us uniform bounds on $ |A|$ and its derivatives up to any order on a neighbourhood of the form $\bar{\mc P}( 0 ,0,d,d)$ for a suitable $ d > 0$. This gives us uniform estimates in $ C^\infty $ on $ \bar F_k$. Hence we can apply Arzela-Ascoli and conclude that there exists a subsequence that converges in $ C^\infty $ to some limit flow which we denote by $ \tilde M_\tau^\infty$. We now analyse the limit flow $ \tilde M_\tau^\infty\subset \mathbb{R}^{n+k}$. Note that we have 
\begin{align*}
\bar A_k ( p , \tau ) = \hat r_k A_k ( p , \hat r_k^2 \tau+t_k ). 
\end{align*}
so that
\begin{align*}
\frac{|\bar A_k(p, \tau) |^2}{|H_k(p,\tau ) |^2}& = \frac{|A_k ( p, \hat r_k^2 \tau+t_k ) |^2}{|H_k(p, \hat r_k^2 \tau+t_k ) |^2} 
\end{align*}
but since $ \hat r_k \rightarrow 0, t_k \rightarrow T_{\max}$ as $ k\rightarrow \infty $ this implies 
\begin{align*}
\frac{|\bar A ( p , \tau ) |^2}{| \bar  H (p, \tau ) |^2}=\lim_{k\rightarrow \infty}\frac{|\bar A_k(p,\tau)|^2}{|\bar H_k ( p, \tau ) |^2} \leq \frac{1}{n} +\e \quad \text{and} \quad  \frac{| \bar A(0,0)|^2}{| \bar H(0,0)|^2}= \frac{1}{n}+\e.
\end{align*}
Hence the flow $\bar{\mc M}_t^\infty $ has a space-time maximum for $\frac{|\bar A ( p , \tau ) |^2}{| \bar  H (p, \tau ) |^2}$ at $ (0,0)$. Since the evolution equation for $ \frac{|A|^2}{|H|^2}$ is given by 
\begin{align*}
\partial_t \left ( \frac{|A|^2}{|H|^2}\right)-\triangle \left (\frac{|A|^2}{|H|^2}\right) & = \frac{2}{|H|^2}\left \la \nabla |H|^2 , \nabla \left ( \frac{|A|^2}{|H|^2}\right) \right \ra-\frac{2}{|H|^2} \left ( |\nabla A|^2-\frac{|A|^2}{|H|^2}|\nabla H|^2 \right) \\
&+\frac{2}{|H|^2}\left( R_1-\frac{|A|^2}{|H|^2} R_2\right)
\end{align*}
Now
\begin{equation*}
|\nabla H|^2 \leq \frac{3}{n+2}|\nabla A|^2, \quad  \frac{|A|^2}{|H|^2}\leq c_n \implies -\frac{2}{|H|^2} \left ( |\nabla A|^2-\frac{|A|^2}{|H|^2}|\nabla H|^2 \right) \leq 0.
\end{equation*}
Furthermore if $\frac{|A|^2}{|H|^2}=c <  c_n$ then 
\begin{align*}
 R_1-\frac{|A|^2}{|H|^2} R_2&=  R_1-c  R_2\\
 &\leq \frac{2}{n}\frac{1}{c-\nicefrac{1}{n}}| A_-|^2 \mc Q+\left(6-\frac{2}{n (c-\nicefrac{1}{n})}  \right) |\circo A_1|^2 | \circo A_-|^2+\left(3-\frac{2}{n (c-\nicefrac{1}{n})}  \right)|\circo A_-|^4\\ 
 &\leq 0.
\end{align*}
Hence the strong maximum principle applies to the evolution equation of $\frac{|A|^2}{|H|^2}$ and shows that $\frac{|A|^2}{|H|^2}$ is constant. The evolution equation then shows that $ |\nabla A|^2= 0$, that is the second fundamental form is parallel and that $|A_-|^2 = |\circo A_-|^2=0$, that is the submanifold is codimension one. Finally this shows that locally $ \mc M = \mbb S^{n-k}\times \R^k$, \cite{Lawson1969}. As $\frac{|A|^2}{|H|^2}< c_n\leq  \frac{1}{n-1}$ we can only have $\mbb S^n$ which gives $\frac{|A|^2}{|H|^2}= \frac{1}{n}$ which is a contradiction. 

Next we consider the case where $ T_{\max}= + \infty$. Firstly we rule out the possibility that $\lim_{t\rightarrow\infty} |H_{\max}| = +\infty$. Therefore let us assume  $\lim_{t\rightarrow\infty} |H_{\max}| = +\infty$.

Since by assumption $\abs{H}_{\text{max}} \rightarrow \infty$ as $t \rightarrow \infty$, there exists a $\tau(\eta)$ such that $e^{\delta_0 t/2}<\eta$ for all $\tau \leq t < \infty$.  Thus $\abs{\nabla H} \leq \eta\abs{H}^{2}_{\text{max}}$ for all $t \geq \tau$.  Fix some $\delta \in(0,1)$ and set $\eta = \frac{\delta(1-\delta)\varepsilon}{\pi}$.  Let $t\in [\tau(\eta),\infty)$, and $x$ be a point with $\abs{H}(x)=\abs{H}_\text{max}$.  Along any geodesic of length $\frac{\pi}{\varepsilon\delta H_\text{max}}$ from $x$, we have $\abs{H}\geq \abs{H}_\text{max}-\frac{\pi}{\varepsilon\delta \abs{H}_\text{max}}\eta\abs{H}^2_\text{max}=\delta\abs{H}_\text{max}$, and consequently the sectional curvatures satisfy $K\geq \varepsilon^2\delta^2\abs{H}_\text{max}^2$. From Bonnet's Theorem it follows that $\diam \Sigma \leq \frac{\pi}{\varepsilon\delta H_\text{max}}$, from which we conclude that $\abs{H}_\text{min}\geq \delta\abs{H}_\text{max}$ on the whole of $\Sigma_t$ for $t\in [\tau(\eta),T)$. 

The previous line shows that by choosing $\tau$ sufficiently large, $\abs{H}_{\text{min}}$ can be made arbitrarily large.  It follows from the above argument that after some sufficiently large time the submanifold is as pinched as we like (and in particular can be made to satisfy $\abs{A}^2 < 1/(n-1)\abs{H}^2$ in dimensions $n \geq 4$ and $\abs{A}^2 < 4/(3n)\abs{H}^2$ in dimensions $2 \leq n \leq 4$).  We now show that once the submanifolds are pinched as such, the maximal time of existence must be finite.  Define $Q=\abs{H}^2 - a\abs{A}^2 - b(t)$, where $a=\frac{3n}{4}$ and $b$ is some time-dependent function.  Because $\abs{H}_{\text{min}} > 0$ and the submanifolds are as pinched as we like, for some sufficiently large time $\tau$ we can choose a $b(\tau) = b_{\tau} >0$ such that $Q\geq 0$ for $t = \tau$.  The evolution equation for $\mathcal{Q}$ is
\begin{align*}
	\frac{\p}{\p t}Q &= \Delta Q - 2(\abs{\nabla H}^2 - a\abs{\nabla A}^2) + 2R_2 - 2aR_1 + 2(n-a)\bar{K}\abs{\ho}^2 + 2an\bar{K}\abs{H}^2 - b'(t) \\
	&\geq \Delta Q - 2(\abs{\nabla H}^2 - a\abs{\nabla A}^2) + 2R_2 - 2aR_1 - b'(t).
\end{align*}
Estimating the reaction terms as before we obtain
\begin{align*}
&2R_2-2aR_1-b'(t) \\
	&\quad= \sum_{i,j} \left( \sum_{\alpha} H_{\alpha}h_{ij\alpha} \right)^2 - 2a\sum_{\alpha, \beta} \bigg( \sum_{i,j} h_{ij\alpha}h_{ij\beta} \bigg)^2 - 2a\abs{\Rp}^2 - b'(t) \\
	&\quad \geq  2\abs{\ho_1}^2(a\abs{\ho_1}^2 + a\abs{\ho_-}^2 + b) + \frac{2}{n(1-a/n)}(a\abs{\ho_-}^2 + b)(a\abs{\ho_1}^2 + a\abs{\ho_-}^2 + b) \\
	&\quad - 2a\abs{\ho_1}^4 - 8a\abs{\ho_1}^2\abs{\ho_-}^2 - 3a\abs{\ho_-}^4 - b'(t).
\end{align*}
Equating coefficients, we find $Q \geq 0$ is preserved if
$\frac{db}{dt} \leq \frac{8b^2}{n}$. We can therefore take
\begin{equation*}
	b(t) = \frac{nb_0}{n-8b_0(t - \tau)}.
\end{equation*}
This is unbounded as $t \rightarrow \tau +\frac{n}{8b_0}$, so we must have $T \leq \tau +\frac{n}{8b_0}$. 

Finally we need to consider the case where $T_{\max}=\infty$. Since $\abs{A}_{\text{max}}$ is bounded, by Theorem \ref{thm_gradient} we have the estimate
\begin{equation}\label{e: grad H bounded est}
	\abs{\nabla A}^2 \leq Ce^{-(\delta_0/2) t}.
\end{equation}
By considerings translations in time $(x,t)\mapsto(x,t-T)$ we can therefore extract a convergent subsequence which will independent of the $T$'s approaching infinity. Furthermore this is a static solution to the mean curvature flow and hence a minimal submanifold, that is $\lim_{t\rightarrow \infty} |H| =0$ and \eqref{e: grad H bounded est} tells us that this has parallel second fundamental form. But since the limit submanifold is static, this means that the nonlinearity in \eqref{e: pinch lem sphere 1} must be zero but this can only happen if $\lim_{t\rightarrow \infty}|A|^2=0$ as required.  
\end{proof}

We now have all the necessary estimates in place to repeat the convergence arguments of \cite{Andrews2010} to obtain smooth convergence of the submanifolds to a totally geodesic submanifold.


\section{Mean Curvature Flow of Codimension Two Surfaces in $\mbb S^4$}\label{s:twoSurfaces}
In the case of surfaces in $\mbb{S}^4$ we consider instead the pinching quantity $\abs{A}^2 +2\gamma|K^{\bot}| \leq k\abs{H}^2+\e\bar{K}$ where $\gamma$ and $\e$ will be determined. This is the first step of the proof of Theorem \ref{thm:MainThm2}.

\subsection{Evolution of normal curvature}
In this section we compute the evolution equation for the normal curvature. The normal curvature tensor in local orthonormal frames for the tangent $\{e _i  : i = 1,2 \}$ and normal $\{\nu _ \alpha : \alpha = 1,2 \}$ bundles is given by 
\begin{equation} \label{evol_normal}
R ^{\perp} _{ij\alpha\beta} = h _{i p \alpha} h _{jp \beta} - h _{jp \alpha} h _{ip \beta}.
\end{equation}
We often compute in a local orthonormal normal frame $\{\nu _ \alpha : \alpha = 1,2 \}$ where $\nu_1 = \nicefrac{H}{| H|}$. As the normal bundle is two dimensional $ \nu_2 $ is then determined by $ \nu _ 1 $ up to sign. With this choice of frame the second fundamental form becomes
\begin{align}\label{eqn_traceless} 
\left \{
\begin{array}{cc} 
\circo A _ 1 = A _ 1 - \frac{|H|}{n} Id \\
\circo A_ 2 = A_2
\end{array}
\right.
\quad 
\&\quad 
\left \{
\begin{array}{cc}
\tr A_1 = | H| \\
\tr A_2 = 0.
\end{array}
\right.
\end{align}
It is also always possible to choose the tangent frame $\{e _i  : i = 1,2 \}$ to diagonalise $A_1$. We often refer to the orthonormal frame $ \{e_1, e_2, e_3, e_4\} = \{e_1, e_2, \nu_1, \nu _2\}$, where $\{e_i\}$ diagonalises $A_1$ and $\nu_1 = H/|H|$, as the `special orthonormal frame'.
Codimension two surfaces have four independent components of the second fundamental form, which still makes it tractable to work with individual components, similar to the role of principal curvatures in hypersurface theory. Working in the special orthonormal frame, we often find it convenient to represent the second fundamental form by
\begin{align}\label{eqn_split}
h_{ij} = \left [
\begin{array}{cc}
 \frac{|H|}{2} + a  & 0 \\
 0 & \frac{|H|}{2} - a 
\end{array}
  \right]\nu_1
  + \left[ 
  \begin{array}{cc} 
  b & c \\
  c & -  b
  \end{array}
  \right]\nu_2,
\end{align}
so that $h_{111} = |H|/2 +a$, $h_{221} = |H|/2-a$, $h_{112} = b$, $h_{122} = c$ and so on. Note that $|\Acirc|^2 = 2a^2 + 2b^2 +2c^2$.

Just as a surface has only one sectional curvature $K$, a codimension two surface also has only one normal curvature, which we denote by $ K ^ \perp$. In the special orthonormal frame the normal curvature is
\begin{equation}\label{eqn:normalCurv}
\begin{split}
K^\perp = R^{\perp}_{1234} & = \sum_{p}\left(h_{1p\3}h_{2p\4} - h_{2p\3}h_{1p\4}\right)\\
& = h_{11\3}h_{21\4} - h_{21\3}h_{11\4} +h_{12\3}h_{22\4} - h_{22\3}h_{12\4}\\
& = 2 a c.
\end{split}
\end{equation}
Note also that $| \Rp | ^ 2 = 16 a ^ 2 c ^ 2 $. The evolution equation of the normal curvature is given by 
\begin{align*}
\frac{\partial}{\partial t} R ^{\perp} _{ij\alpha \beta} &= \Delta R ^{\perp} _{ij \alpha \beta} - 2 \sum _{p,r}\left ( \nabla_q h _{ip\alpha} \nabla _ q h _{jp\beta} - \nabla _ q h _{jp \alpha} \nabla _ q h _{i p \beta}\right ) \\
 &\quad +\sum_{p} \left (\frac{d}{dt} h _{i p \alpha} h _{j p \beta} + h _{i p \alpha} \frac{d}{dt} h _{j p \beta} - \frac{d}{dt} h _{j p \alpha} h _{i p \beta} - h _{jp \alpha} \frac{d}{dt}h_{i p \beta}\right)\\
& - n \bar K R_{ij\alpha \beta}^\perp\\
\end{align*}
or 
\begin{equation}\label{eqn:evolNormalCurv}
\begin{split}
\frac{\partial}{\partial t}   R ^{\perp} _{ij\alpha \beta} & = \Delta R ^{\perp} _{ij \alpha \beta} - 2 \sum _{p,r}\left ( \nabla_q h _{ip\alpha} \nabla _ q h _{jp\beta} - \nabla _ q h _{jp \alpha} \nabla _ q h _{i p \beta}\right ) \\
&\quad + \sum (h _{ip \gamma} \cdot h_{rq\gamma} h _{rq\al} + h _{i q \gamma} \cdot h _{q r \gamma} h _{rp \al}+ h _{p q\gamma} \cdot h_{qr\gamma} h _{ri \al} - 2 h _{ir \gamma} \cdot h _{pq\gamma} h _{rq\al})h_{jp\beta}\\
&\quad+ \sum h_{ip\alpha}(h _{jp \gamma} \cdot h_{rq\gamma} h _{rq\beta} + h _{j q \gamma} \cdot h _{q r \gamma} h _{rp \beta}+ h _{p q\gamma} \cdot h_{qr\gamma} h _{rj \beta} - 2 h _{jr \gamma} \cdot h _{pq\gamma} h _{rq\beta}) \\
&\quad-\sum ( h _{jp \gamma} \cdot h_{rq\gamma} h _{rq\al} + h _{j q \gamma} \cdot h _{q r \gamma} h _{rp \al}+ h _{p q\gamma} \cdot h_{qr\gamma} h _{rj \al} - 2 h _{jr \gamma} \cdot h _{pq\gamma} h _{rq\al}) h_{ip \beta}\\
&\quad - \sum h_{jp\al}(h _{ip \gamma} \cdot h_{rq\gamma} h _{rq\beta} + h _{i q \gamma} \cdot h _{q r \gamma} h _{rp \beta}+ h _{p q\gamma} \cdot h_{qr\gamma} h _{ri \beta} - 2 h _{ir \gamma} \cdot h _{pq\gamma} h _{rq\beta} )\\
&- n \bar K R_{ij\alpha \beta}^\perp\\
&=\Delta R ^{\perp} _{ij \alpha \beta} - 2 \sum _{p,r}\left ( \nabla_q h _{ip\alpha} \nabla _ q h _{jp\beta} - \nabla _ q h _{jp \alpha} \nabla _ q h _{i p \beta}\right ) \\
&\quad+R_{ip\alpha}h_{jp\beta} + h_{ip\alpha}R_{jp\beta} - R_{jp\alpha}h_{ip\beta} - h_{jp\alpha}R_{ip\beta}- n \bar K R_{ij\alpha \beta}^\perp\\
\end{split}
\end{equation}
Computing in the special orthonormal frame and denoting the reaction terms by $ \frac{d}{dt}  K ^ \perp$, the nonlinearity for codimension two surfaces simplifies to
\begin{align*}
 \frac{d}{dt} K ^ \perp &= 4 ac \left(  \left ( \frac{| H|}{2} - a \right)^ 2 - \left (\frac{|H|}{2} + a  \right) \left ( \frac{| H|}{2} - a \right) + 2 b ^ 2 + 3 c ^ 2 +\left ( \frac{|H|}{2}  + a \right) ^ 2\right) \\
  &= K^{\perp}\left( |A|^2 + 2|\Acirc|^2 - 2b^2 \right)-2\bar K K^\perp. 
\end{align*}
For notational convenience we set
\[ \nabla_{\!\!evol}K^{\perp} := \sum _{p,q}\left ( \nabla_q h _{ip\alpha} \nabla _ q h _{jp\beta} - \nabla _ q h _{jp \alpha} \nabla _ q h _{i p \beta}\right ) \quad 
\text{and} \quad R_3 := K^{\perp}\left( |A|^2 + 2|\Acirc|^2 - 2b^2 \right). \]
Substituting the simplifed nonlinearity into \eqref{eqn:evolNormalCurv} we obtain the evolution equation for the normal curvature
\[ \frac{\p }{\p t} K^{\perp}= \Delta K^{\perp} - 2 \nabla_{\!\!evol}K^{\perp} + K^{\perp}\left( |A|^2 + 2|\Acirc|^2 - 2b^2 \right)-2\bar K K^\perp,\]
and a little more computation shows the length of the normal curvature evolves by
\[ \frac{\p }{\p t} |K^{\perp}|= \Delta |K^{\perp}| - 2 \frac{K^{\perp}}{|K^{\perp}| }\nabla_{\!\!evol}K^{\perp} + |K^{\perp}|\left( |A|^2 + 2|\Acirc|^2 - 2b^2 \right) - 2 \bar K |K^\perp|. \]
We remark that the complicated structure of the gradient terms prevents an application of the maximum principle to conclude flat normal normal bundle is preserved.

\begin{proposition}\label{prop_grad}
We have the following gradient estimates:
\begin{subequations}
\begin{align}
|\nabla A | ^ 2 &\geq \frac {3}{n+2} | \nabla H | ^ 2  \label{eqn_gradient1} \\
|\nabla A | ^ 2 - \frac {1} {n} |\nabla H| ^ 2 &\geq \frac {2 ( n-1)}{3 n} | \nabla A | ^ 2 \label{eqn_gradient2}\\
|\nabla A| ^ 2 & \geq  2 \nabla_{\!\!evol} K^{\perp} \quad \text{if $ n =2 $}. \label{eqn_gradient3}
\end{align}
\end{subequations}
\end{proposition}
\begin{proof} 
The first two inequalities are proven in \cite{Huisken1984}, motivated by similar estimates in the Ricci flow \cite{Hamilton1982}. They are established by decomposing the tensor $ \nabla A $ into orthogonal components $\nabla _i  h _{jk} = E _{ijk} + F _{ijk}$,
where 
\begin{align*}
 E_{ijk} = \frac {1}{n+2} ( g _{ij} \nabla _k H + g_{ik} \nabla _j H + g_{jk} \nabla _i H ),
\end{align*}
from which it follows that $ |\nabla A | ^ 2 \geq | E | ^ 2 = \frac{3}{n+ 2} | \nabla H | ^ 2 $. The second estimate follows from the first. In order to prove the third inequality, we evaluate directly 
\begin{multline*}
\sum _{p,q}\left ( \nabla_q h _{1p1} \nabla _ q h _{2p2} - \nabla _ q h _{2p 1} \nabla _ q h _{1p 2}\right )
= \nabla _{1} h _{111} \nabla_1 h_{212} - \nabla _ 1 h_{211} \nabla _ 1 h _{112} +  \nabla _1 h _{121} \nabla _ 1 h _{222}  \\
-  \nabla _ 1 h _{221} \nabla _ 1 h _{122} + \nabla_2h_{111}\nabla_2 h_{212}- \nabla_2 h_{211}\nabla_2 h_{112} +\nabla _ 2 h _{121} \nabla _2 h _{222} - \nabla _ 2 h _{221} \nabla _ 2 h _{122}.
\end{multline*}
Writing down all the terms in $ |\nabla A | ^ 2 $ and only using the symmetries of the second fundamental form
\begin{align*}
|\nabla A| ^ 2 & =( \nabla _ 1 h _{111} )^ 2 + ( \nabla _ 2 h _{111} )^ 2 + (\nabla_1 h_{121})^2+ (\nabla_1  h_{211})^2 + ( \nabla _ 1 h _{122} )^ 2 + (\nabla_1 h_{212})^2 +(\nabla_2 h_{112})^2 \\
&+ ( \nabla _{1} h _{222} )^ 2 +(\nabla_{2} h_{122})^2 + (\nabla _{2} h_{212})^2+( \nabla _2 h_{221} )^ 2 + (\nabla _ 2 h _{222} )^ 2\\
& +  (\nabla _ 1 h _{221})^ 2 + ( \nabla_2 h_{121})^2 + ( \nabla _1 h_{211})^2  +(\nabla_1 h_{112}) ^2,
\end{align*}
and the estimate follows by applying the Cauchy-Schwarz inequality and comparing terms.
\end{proof}

We consider here the pinching quantity
\begin{align*}
 \mc Q := | A| ^ 2 +2 \s | K ^ \perp|  - k |H| ^ 2 - \epsilon\bar K< 0
\end{align*}
The evolution equation becomes
\begin{align*}
\frac{\partial}{\partial t} \mc{Q} &= \Delta \mc{Q} - 2 \left( |\nabla A | ^ 2 + 2\gamma \frac{\Kp}{|\Kp|}\nabla_{evol} \Kp - k|\nabla H| ^ 2 \right) \\ &+ 2 R_1 + 2\gamma R_3 - 2 kR  _2 -4\bar K|\circo A|^2 + 2 \bar K |H|^2 - 4 k\bar K |H|^2  -4 \gamma \bar K |K^\perp| 
\end{align*}
We deal with the gradient terms first. Using the the gradient estimates \eqref{eqn_gradient1} and \eqref{eqn_gradient3} we have
\begin{align*}
-2 \left( |\nabla A | ^ 2 + 2\gamma\frac{\Kp}{|\Kp|}\nabla_{evol} \Kp - k |\nabla H| ^ 2 \right) &\leq \left(- 2 + 2\gamma + 2\frac43k \right) |\nabla A|^2,
\end{align*}
which is less than zero provided $\gamma < (1-4/3k)$.

Next we deal with the reaction terms
\begin{align}
\frac{d}{dt} \mc Q &= 2\sum _{\alpha ,\beta}\bigg(\sum_{i, j} h _{ij\alpha} h_{ij \beta} \bigg)^ 2  + 2| \Rp  | ^ 2 - 2 k \sum_{i, j} \left ( \sum_{\alpha} H _ \alpha h_{ij \alpha}\right ) ^ 2 + 2\gamma R_3   \nonumber \\
&= 2 | \Acirc_1 | ^ 4  - 2 \left( k - \frac{2}{n} \right) | \Acirc_1 | ^ 2 | H| ^ 2 - \frac 2 n \left( k - \frac{1}{n} \right) | H| ^ 4 \nonumber \\
&\quad + 4 \bigg( \sum_{i,j} \circo h _{ij1}\circo  h _{ij2}\bigg)^2  + 2 \bigg( \sum_{i ,j} \circo h_{ij 2} \circo h_{ij 2} \bigg) ^ 2 + 2 | \Rp | ^ 2 \nonumber \\
&\quad + 2\gamma|K^{\perp}|\left( |A|^2 + 2|\Acirc|^2 - 2b^2 \right)-4\bar K|\circo A|^2 -4 \bar K(k-\nicefrac12) |H|^2 -4 \gamma \bar K |K^\perp|. \label{eqn:Qrxn}
\end{align}

Written in the special orthonormal frame, the bracketed terms on the second last line above are
\begin{align*}
4  \bigg( \sum_{i,j} \circo h _{ij1}\circo  h _{ij2}\bigg)^2 & = 16 a ^ 2 b^ 2, \quad 2 \bigg( \sum_{i ,j} \circo h_{ij 2} \circo h_{ij 2} \bigg) ^ 2 =  2 ( 2 b ^ 2 + 2 c ^ 2 ) ^ 2. 
\end{align*}
Now suppose, for a contradiction, that there exists a first point in time where $\mc Q = 0$. Computing at this point, as $\mc Q = 0$ we have $ \left( k - \frac 1 n \right) | H |^ 2 = (|\Acirc |^ 2 + 2 \s  | K^ \perp | -\e\bar K)$, and substituting this into \eqref{eqn:Qrxn} to eliminate the $|H|^2$ terms we obtain after some computation
\begin{equation}\label{eqn:Qrxn}
\begin{split}
\frac{d}{dt} \mc Q &= \left(-\frac{1}{k-1/2} + 2 \right)4 a^2b^2 + \left(-\frac{1}{k-1/2} + 2 \right) \gamma|\Kp||\Acirc_1|^2  \\
&\quad +  \left(-\frac{3}{k-1/2} + 6 \right) \gamma|\Kp||\Acirc_2|^2 + \left(-\frac{1}{k-1/2} + 2 \right)|\Acirc_2|^4 \\
&\quad + \left(-\frac{(1 + 2\gamma^2)}{k-1/2} + 6 \right)|\Kp|^2 \\
&\quad + \epsilon\bar K\left( 2 + \frac{1}{k-1/2} \right) |\Acirc_1|^2 + \frac{2\epsilon \bar K}{k-1/2}|\Acirc_2|^2 + \frac{3\epsilon \bar K \gamma|\Kp|}{k-1/2} - \frac{\epsilon^2\bar K ^2}{k-1/2}\\
&-4\bar K|\circo A|^2 -4 \bar K(|\Acirc |^ 2 + 2 \s  | K^ \perp | -\e\bar K) -4 \gamma \bar K |K^\perp|
\end{split}
\end{equation}
The quartic terms are 
\begin{align*}
 &\left(-\frac{1}{k-1/2} + 2 \right)4 a^2b^2 + \left(-\frac{1}{k-1/2} + 2 \right) \gamma|\Kp||\Acirc_1|^2  \\
&\quad +  \left(-\frac{3}{k-1/2} + 6 \right) \gamma|\Kp||\Acirc_2|^2 + \left(-\frac{1}{k-1/2} + 2 \right)|\Acirc_2|^4 + \left(-\frac{(1 + 2\gamma^2)}{k-1/2} + 6 \right)|\Kp|^2
\end{align*}
and in the special orthonormal frame these are 
\begin{align*}
&4c^2 \left\{\left(-\frac{1}{k-1/2} + 2 \right)c^2 + \eta_1\left(-\frac{3}{k-1/2} + 6 \right) \gamma |ac|
 + \eta_2 \left(-\frac{(1 + 2\gamma^2)}{k-1/2} + 6 \right) a^2 \right\} \\
&\quad + 4|ac| \bigg\{\left(-\frac{1}{k-1/2} + 2 \right)\gamma a^2 + (1-\eta_2)\left(-\frac{(1 + 2\gamma^2)}{k-1/2} + 6 \right) |ac| \\
&\quad + (1-\eta_1)\left(-\frac{3}{k-1/2} + 6 \right) \gamma c^2 \bigg\}.
\end{align*}
We now substitute $\gamma = 1 - 4/3 k - \delta$ in order to keep the gradient term negative, and use the parameters $\eta_1, \eta_2$ to shift as much bad normal curvature into the first curly bracket to consume all of the good $c^4$ term. As it does not seem possible to reach $k = 3/4$, we have numerically explored the parameter values, with the result that the above term is strictly negative for $k = 29/40$.

The lower order terms are 
\begin{align*}
 &\epsilon\bar K\left( 2 + \frac{1}{k-1/2} \right) |\Acirc_1|^2 + \frac{2\epsilon \bar K}{k-1/2}|\Acirc_2|^2 + \frac{3\epsilon \bar K \gamma|\Kp|}{k-1/2} - \frac{\epsilon^2\bar K ^2}{k-1/2}\\
&-4\bar K|\circo A|^2 -4 \bar K(|\Acirc |^ 2 + 2 \s  | K^ \perp | -\e\bar K) -4 \gamma \bar K |K^\perp|
\end{align*}
Rearranging we have 
\begin{align*}
&\left(\left( 2 + \frac{1}{k-1/2} \right)\e - 8\right)\bar K |\circo A_1|^2 + \left(\frac{2\epsilon}{k-1/2}-8 \right)\bar K |\circo A_2|^2+ 3\gamma\bar K \left( \frac{\epsilon}{k-1/2} -4 \right)|\Kp|\\
&+\left(4   - \frac{\epsilon}{k-1/2}\right)\e \bar K^2
\end{align*}
The last two terms are zero if $\left(4   - \frac{\epsilon}{k-1/2}\right)= 0$. Therefore we require $\e = 4( k - \frac 1 2)$ and $\gamma \geq 0 $. We also require  $\left( 2 + \frac{1}{k-1/2} \right)\e - 8 \leq 0$, $\left(\frac{2\epsilon}{k-1/2}-8 \right)\leq0$ but this occurs if $\e \leq 2$ which is implied if $ k \leq \frac{29}{40}$ (in fact if $k \leq 1$).

\section{Minimal Submanifolds of the Sphere}
Minimal surfaces are geometric obstructions to enlarging preserved curvature cones. One minimal surface of particular relevance to the mean curvature flow in a sphere is the Clifford torus, which is a minimal in $\mathbb S^3$ and satisfies $|A|^2 = |H|^2$ when immersed into $\mathbb R^4$. For two surfaces immersed in the three-sphere, the Clifford torus is a geometric obstruction to pushing the pinching condition beyond $1/(n-1)$. However, the mean curvature flow is currently unable to reach the Clifford torus due to technical problems with the gradient terms (see \cite{Andrews2002} where this problem is overcome by a fully nonlinear flow). We speculate the geometric obstruction to two surfaces in $\mathbb S^4$ evolving by the mean curvature flow is the Veronese surface, which is minimal in $\mathbb S^4$ and satisfies $|A = \nicefrac{5}{6}|H|^2$ when immersed in $\mathbb R^5$. In this final section, we refine a famous theorem by Chern, do Carmo and Kobayashi by charactering minimal surfaces of the four-sphere in terms of a pointwise pinching of intrinsic and extrinsic curvatures, instead of the length of the second fundamental form as was done in their original paper \cite{ChernCarmoKobayashi1970}. This is achieved by exact calculation of the the nonlinearity in the Simons identity. The equivalent result for a Euclidean background appeared in \cite{Baker2017}, where it was used to greatly simplify the proof the Poincar\'e-type inequality obtained from the positivity of the Simons identity nonlinearity. We first compute the Simons identity nonlinearity exactly and then achieve the desired result by an application of the strong maximum principle.

\begin{prop}

Let $ \Sigma^2\subset \mbb S^4$. Then the contracted Simons' identity has the form
\begin{align*}
\frac12 \triangle |A| ^ 2 & = A_{ij} \cdot \nabla_i\nabla_j H+ |\nabla A|^2 +  ( H^2 - |A| ^ 2 + 2\bar K ) | \circo A|^2 - 2 | K ^ \perp | ^ 2\\
& =A_{ij} \cdot \nabla_i\nabla_j H+ |\nabla A| ^2 + K |\circo A|^2 - 2 |K^\perp|^2 
\end{align*}

\end{prop}
\begin{proof} 
The contracted Simons' identity takes the form
\begin{equation}\label{e: contracted Simons' identity}
	\frac{1}{2} \Delta \abs{A}^2 = A_{ij} \cdot \nabla_i\nabla_j H + \abs{\nabla A}^2 + Z + 2\bar K\abs{\ho }^2,
\end{equation}
where again
\begin{equation*}
	Z = -\sum_{\alpha, \beta} \bigg( \sum_{i,j} h_{ij\alpha} h_{ij\beta} \bigg)^2 - \abs{\Rp}^2 + \sum_{\substack{i,j,p \\ \alpha, \beta}} H_{\alpha} h_{ip\alpha} h_{ij\beta} h_{pj\beta}.
\end{equation*}
Splitting the first term on the right into diagonal and off-diagonal summations, and using $h_{ij1} = 0$ for $i \neq j$, we get
\begin{align*}
\sum_{i,j,p,\alpha,\beta} H_\alpha h _{ip\alpha} h _{ij \beta} h _{pj \beta} &=   \sum_i h_{ii\alpha} \sum_{i,j} h_{ii\alpha} (h_{ii1})^2  + \sum_i h_{ii\alpha} \sum_{i,j} h_{ii\alpha} (h_{ii2})^2  \\
&\quad +  \sum_i h_{ii\alpha} \sum_{i \neq j} h_{ii\alpha} (h_{ij2})^2   + \sum_i h_{ii\alpha} \sum_{i \neq p} h_{ip\alpha} h_{ij\beta}h_{pj\beta}.
\end{align*}
The final term on the right is zero, as computing in the special orthonormal frames we see
\begin{align*}
\sum_i h_{ii\alpha} \sum_{i \neq p} h_{ip\alpha} h_{ij\beta}h_{pj\beta}&= H \, \sum_{i \neq p} h_{ip1} h_{ij\beta}h_{pj\beta} =0,
\end{align*}
since $h_{ip1} = 0$ for $i \neq p$.
We similarly split the second term on the right of Z into diagonal and off-diagonal sums, and putting all terms together we have
\begin{align*}
Z &= \sum_i h_{ii\alpha}  \sum_{i,j} h_{ii\alpha} (h_{ii1})^2 + \sum_i h_{ii\alpha}  \sum_{i,j} h_{ii\alpha} (h_{ii2})^2  + \sum_i h_{ii\alpha} \sum_{i \neq j} h_{ii\alpha} (h_{ij2})^2 \\ 
&\quad - \sum_{\alpha} \bigg( \sum_{i} h_{ii1} h_{ii\alpha} \bigg)^2 - \sum_{\alpha} \bigg( \sum_{i} h_{ii2} h_{ii\alpha} \bigg)^2 - \sum_{\alpha} \bigg( \sum_{i\neq j} h_{ij2} h_{ij\alpha} \bigg)^2\\ &\quad-2\sum_{\alpha,\beta}\bigg( \sum_{i=j} h _{ij\alpha}h_{ij\beta} \sum_{i\neq j}h_{ij\alpha}h_{ij\beta} \bigg)
 - |\Rp| ^ 2.\\
\end{align*}
We estimate these terms in pairs, gathering the first, second and third terms of lines one and two, respectively. Dealing with the first pair of terms, we follow \cite{Smyth1973} but keep track of the normal curvature terms to find
\begin{align*}
\sum_i h_{ii\alpha} \sum_{i,j} h_{ii\alpha} (h_{ii1})^2  - \sum_{\alpha} \bigg( \sum_{i} h_{ii1} h_{ii\alpha} \bigg)^2 &= \left( \left(|H|^2 - |A|^2\right) + \sum_{\alpha}(h_{12\alpha})^2 \right)(h_{111} - h_{221})^2 \\
&= \left(|H|^2 - |A|^2\right)(4a^2) + 4a^2c^2.
\end{align*}
We estimate the second pair of terms in the same way, obtaining
\begin{align*}
\sum_i h_{ii\alpha}  \sum_{i,j} h_{ii\alpha} (h_{ii2})^2 -  \sum_{\alpha} \bigg( \sum_{i} h_{ii2} h_{ii\alpha} \bigg)^2 &= \left( \left(|H|^2 - |A|^2\right) + \sum_{\alpha}(h_{12\alpha})^2 \right)(h_{112} - h_{222})^2 \\
&= \left(|H|^2 - |A|^2\right)(4b^2) + 4b^2c^2.
\end{align*}
For the third pair of terms, as there are no diagonal terms to easily factor into the intrinsic curvature, we proceed by computing in the special orthonormal frames from the outset:
\begin{align*}
\sum_i h_{ii\alpha}  \sum_{i \neq j} h_{ii\alpha} (h_{ij2})^2  - \sum_{\alpha} \bigg( \sum_{i\neq j} h_{ij2} h_{ij\alpha} \bigg)^2 &= 4c^2 \left(\frac{|H|^2}{4} - c^2\right) \\
&= 4c^2 \left( \frac{|H|^2}{4} - a^2 - b^2 - c^2 \right) + 4c^2(a^2 + b^2) \\
&= 2c^2 (|H|^2 - |A|^2) + 4c^2(a^2 + b^2).
\end{align*}
With the final term, as $h_{ij1} = 0$ the only non-zero contribution comes from $\alpha,\beta =2$ and we see
\begin{align*}
2\sum_{\alpha,\beta}\bigg( \sum_{i=j} h _{ij\alpha}h_{ij\beta} \sum_{i\neq j}h_{ij\alpha}h_{ij\beta} \bigg) &= 2\bigg( \sum_{i=j} h _{ij2}h_{ij2} \sum_{i\neq j}h_{ij2}h_{ij2} \bigg)\\
&=2(2 b^2)(2c^2) = 8 b^2c^2.
\end{align*}
Collecting all the terms together, and recalling $|\Rp|^2 = 16a^2c^2 = 4|K^\perp|^2$, we achieve
\begin{align*}
Z &= \left(|H|^2 - |A|^2\right) (2a^2 + 2b^2 + 2c^2) + 8a^2c^2 + 8b^2c^2 - 16a^2c^2-8 b^2c^2\\
&= \left (|H|^2 - |A|^2 \right)|\Acirc|^2 - 2|K^\perp|^2.
\end{align*}
\end{proof}

We now apply the above proposition in the case $\Sigma^2$ is a minimal surface to conclude theorem \ref{thm:MainThm3}.

\begin{proof}[Proof of theorem \ref{thm:MainThm3}]
Suppose $\Sigma^2$ is minimally immersed in $\mathbb S^4$. Then the nonlinearity in the Simons identity satifies

\begin{align*}
 Z &= - |A|^ 4 - 2 | K ^ \perp| ^ 2 +2 |A|^ 2 = | A| ^ 2 \left ( - |A|^2 - \frac{| K^\perp|^2}{|A|^2} + 2  \right).
\end{align*}
Therefore let us assume $ - |A|^2 -\frac{2 |K^\perp|^2}{|A|^2} + 2 \geq  0$ or equivalently, $2K \geq \frac{2 |K^\perp|^2}{| A|^2}$
so that $ K^\perp= 2 ac \leq a ^ 2 + b ^ 2 + c ^ 2 \leq \frac 12 |A|^2$. Then computing the contracted Simons' identity we have 
\begin{align}\label{eqn_contracted_simons}
\frac 12 \triangle |A| ^ 2 & = |\nabla A | ^ 2 + |A|^2 \left ( - |A|^2 - \frac{2 |K^\perp| ^ 2}{|A|^ 2} +2  \right)
\end{align}
Therefore if we have $ - |A|^2 -\frac{2 |K^\perp|^2}{|A|^2} + 2 \geq  0$ this implies $ \frac 12 \triangle |A|^2 \geq 0$ so by the maximum principle $ |A|^2 \equiv C = const.$ By \eqref{eqn_contracted_simons}, either $ |A| ^ 2 \equiv 0$ or $|A|^ 2 + \frac{2 |K^\perp| ^ 2}{|A|^2}  =2 $. In both cases we have $ |\nabla A|^ 2 =0$. Since we have $ |A|^2 = const. \implies K^\perp = const.$  

Furthermore as 
$
-|A|^ 2 - \frac{2 |K^\perp| ^ 2}{|A|^2}  =2 
$
implies $ |A| ^ 2 = 1 \pm \sqrt{1 - 2 | K^\perp|^2}.$ Since $ K ^ \perp = 2 ac$ we get 
\begin{align*}
0 \leq |K^\perp|^2 \leq \frac 14 |A|^ 4 
\end{align*}
or $ \frac 43 \leq |A|^ 2 \leq 2$ because $|A|^2 ( 2 - |A|^2) = 2 | K ^ \perp|^2$ so that $\frac 1 4 |A|^4 \geq |A|^2 ( 2 - |A|^2) \geq 0.$
Furthermore this implies $ 0 \leq |K^\perp |^2 \leq \frac 43$. As $ 2 K  = 2 - |A|^2$
gives 
\begin{align*}
2 \int_{\Sigma} Kd \mu = 2 \frac{K^\perp}{|A|^2} \int _{\Sigma} K^\perp d \mu
\end{align*} 
which implies if $ K^\perp =0$ then $ K = 0$ or that $ |A|^2 = 2$ in which case we have a Clifford torus. Therefore let us assume $ K^\perp \neq 0$. We apply Simons' identity to $ K^\perp$.

For the sphere, where $ \bar R _{ijkl} = \bar K ( g_{ik} g _{jl} - g _{il} g _{jk}), \nabla \bar R =0$.

\begin{align}\label{eqn:SimonsId}
\Delta h_{ij\alpha}&=\nabla_i\nabla_jH_\alpha+H\cdot h_{ip}h_{pj\alpha}-h_{ij}\cdot h_{pq}h_{pq\alpha}+2h_{jq}\cdot h_{ip}h_{pq\alpha}-h_{iq}\cdot h_{qp}h_{pj\alpha}-h_{jq}\cdot h_{qp}h_{pi\alpha}\\
\nonumber&+H_\alpha R_{i\alpha j\beta}\nu _\beta -h_{ij\alpha} \bar R_{kjkp}h_{pi \alpha} + \bar R_{kikp} h_{pj\alpha} - 2 \bar R _{ipjp} h_{pq\alpha} - \bar h_{ij\alpha}\bar R_{ip\alpha \beta} \nu _\beta + 2 h_{ip\alpha} \bar R_{jp\alpha\beta} \nu _\beta \\
\nonumber& +\bar \nabla_k \bar R_{kij\beta} \nu _\beta - \nabla _i \bar R_{jkk\beta} \nu_ \beta\\
\nonumber &= \nabla_i\nabla_j H + Z_{ij\alpha} + 2 h_{ij\alpha} - g_{ij} H_{\alpha}
\end{align}
where $$Z_{ij\alpha} = H\cdot h_{ip}h_{pj\alpha}-h_{ij}\cdot h_{pq}h_{pq\alpha}+2h_{jq}\cdot h_{ip}h_{pq\alpha}-h_{iq}\cdot h_{qp}h_{pj\alpha}-h_{jq}\cdot h_{qp}h_{pi\alpha}.$$
Therefore computing in the special orthonormal frames above, with $ H=0$ we get
\begin{align*}
\triangle K^\perp = 2 K^\perp ( 2 - b ^ 2 - 3 a^2 - 3 c ^  2) =0.
\end{align*}
because we can show that $ |K^\perp|^2 = 4 a^2c^2$ and $ |A^+|^2 = 2 a ^2$. Hence $ \frac{ |K^\perp|^2}{|A^+|^2}= 2 c ^2$. Also $ |A^-|^2 = 2 b^2+ 2 c^2$ so that $2b^2 = |A^-|^2 - \frac{ |K^{\perp}|^2}{|A^+|^2}$.  
\begin{align*}
\Delta K^\perp = 2 K^\perp( 2 +  |A^-|^2 - \frac{|K^\perp|^2}{| A^+|^2} - \frac{3}{2} |A|^2) =0.   
\end{align*}

Since we assume $ K ^ \perp \neq 0$ , $2+ 2 b^2 - 3b^2 - 3 a ^ 2 - 3 c ^ 2 =0 $ or $2 +  |A^-|^2 - \frac{|K^\perp|^2}{| A^+|^2} - \frac{3}{2} |A|^2=0$. Therefore 
\begin{align*}
2 b ^ 2 = \frac 32 |A|^ 2 - 2 \quad \text{or} \quad |A^-|^2 - \frac{|K^\perp|^2}{| A^+|^2}  = \frac{3}{2}|A|^2- 2  
\end{align*}
Therefore $ b^2 =const.$ or $|A^-|^2 - \frac{|K^\perp|^2}{| A^+|^2} = const.$. We compute the Laplacian of $ b^ 2 = h_{112}^2$ and get 
\begin{align*}
\triangle h_{112} = h_{112}( 2 -2 a ^ 2 -2 b ^ 2 -2 c ^ 2 )=0.
\quad \text{or} \quad  
\Delta |A^-|^2 - \frac{|K^\perp|^2}{| A^+|^2} = |A^-|^2 - \frac{|K^\perp|^2}{| A^+|^2} ( 2 - |A|^2).
\end{align*}
Therefore $ b=0, |A^-|^2 - \frac{|K^\perp|^2}{| A^+|^2} =0 $ or $ |A|^ 2 =2$. If $|A|^ 2 =2$, then we have the Clifford torus. Therefore let us assume $ b =0, |A^-|^2 - \frac{|K^\perp|^2}{| A^+|^2} =0 $. Then from above $ |A|^ 2 = \frac 43$ and $ K^\perp = \frac 23$.  Then by a theorem of Chern-do Carmo-Kobayashi \cite{ChernCarmoKobayashi1970}, this surface is the Veronese surface. 
\end{proof}


\end{document}